\documentclass[draft]{amsart}

\usepackage{amsthm}
\usepackage{amssymb} 
\usepackage{enumerate} 

\usepackage{hyperref}

\numberwithin{equation}{section}

\theoremstyle{plain}
\newtheorem{lemma}{Lemma}[section]
\newtheorem{theorem}[lemma]{Theorem}
\newtheorem{proposition}[lemma]{Proposition}
\newtheorem{corollary}[lemma]{Corollary}
\newtheorem{question}[lemma]{Question}

\theoremstyle{definition} 
\newtheorem{definition}[lemma]{Definition}
\newtheorem{example}[lemma]{Example}

\theoremstyle{remark}
\newtheorem{remark}[lemma]{Remark} 

\renewcommand{\dim}{\operatorname{dim}}
\newcommand{\depth}{\operatorname{depth}}
\newcommand{\Supp}{\operatorname{Supp}}

\newcommand{\Ext}{\operatorname{Ext}}
\newcommand{\Min}{\operatorname{Min}}
\newcommand{\Ass}{\operatorname{Ass}}
\newcommand{\hh}{\operatorname{H}}

\newcommand{\ee}{{}^{e}\!}
\newcommand{\ph}{{}^{\phi}\!}
\newcommand{\Hom}{\operatorname{Hom}}
\newcommand{\susp}{\mathsf{\Sigma}}
\newcommand{\id}{\operatorname{id}}

\newcommand{\fd}{\operatorname{fd}}
\renewcommand{\le}{\leqslant}
\renewcommand{\ge}{\geqslant}

\newcommand{\kos}[2]{\operatorname{K}[#1;#2]}

\newcommand{\lol}{\ell\ell}
\newcommand{\Spec}{\operatorname{Spec}}
\newcommand{\Tor}{\operatorname{Tor}}
\newcommand{\pd}{\operatorname{pd}}

\newcommand{\bsy}{\boldsymbol{y}}

\newcommand{\vf}{\varphi}

\newcommand{\RHom}{\operatorname{\mathsf{R}Hom}}

\newcommand{\lotimes}{\otimes^{\mathbf L}}

\newcommand{\sfD}{\mathsf D}

\newcommand{\bsx}{\boldsymbol{x}}

\newcommand{\fm}{\mathfrak{m}} 
\newcommand{\fp}{\mathfrak{p}}
\newcommand{\fn}{\mathfrak{n}}

\begin{document}

\title[Characterizations of Gorenstein rings]{Characterizing Gorenstein rings using contracting endomorphisms}

\dedicatory{Dedicated to Craig Huneke on the occasion of his 65th birthday.}
\author[B. Falahola]{Brittney Falahola}

\address{Stephen F. Austin State University, Nacogdoches, TX 75962, U.S.A.}

\email{falaholabl@sfasu.edu}

\urladdr{http://www.sfasu.edu/math}

\author[T.\ Marley]{Thomas Marley}

\address{University of Nebraska-Lincoln, Lincoln, NE 68588, U.S.A.}
\email{tmarley1@unl.edu}

\urladdr{http://www.math.unl.edu/~tmarley1}

\date{\today}

\bibliographystyle{amsplain}



\begin{abstract} 
We prove several characterizations of Gorenstein rings in terms of vanishings of derived functors of certain modules or complexes whose scalars are restricted via contracting endomorphisms.  These results can be viewed as analogues of results of Kunz (in the case of the Frobenius) and Avramov-Hochster-Iyengar-Yao (in the case of general contracting endomorphisms).
 \end{abstract}

\maketitle

\section{Introduction}  

 In 1969 Kunz \cite{Ku} proved that a commutative Noetherian local ring of prime characteristic is regular if and only if some (equivalently, every) power of the  Frobenius endomorphism is flat.    Subsequently,  the Frobenius map has been employed to great effect to study homological properties of commutative Noetherian local rings;  see \cite{PS}, \cite{He}, \cite{R}, \cite{KL}, \cite{AM} and \cite{AHIY}, for example.  In this paper, we explore characterizations of Gorenstein rings using the Frobenius map, or more generally, contracting endomorphisms.   Results of this type have been obtained by Iyengar and Sather-Wagstaff \cite{ISW}, Goto \cite{G}, Rahmati \cite{Ra}, Marley \cite{M}, and others.  Our primary goal is to obtain characterizations for a local ring to be Gorenstein in terms of one or more vanishings of derived functors in which one of the modules is viewed as an $R$-module by means of a contracting endomorphism (i.e., via restriction of scalars).  A prototype of a result of this kind for regularity is given by  Theorem \ref{ahiy} below.
 
Let $R$ be a commutative Noetherian local ring with maximal ideal $\fm$ and $\phi:R\to R$ an endomorphism which is contracting, i.e., $\phi^i(\fm)\subseteq \fm^2$ for some $i$.  For an $R$-module $M$, let $\ph M$ denote the abelian group $M$ viewed as an $R$-module via $\phi$; i.e., $r\cdot m:=\phi(r)m$ for $r\in R$ and $m\in \ph M$.    A far-reaching generalization of Kunz's result due to Avramov, Hochster, Iyengar and Yao \cite{AHIY} states that if there exists a nonzero finitely generated $R$-module $M$ such that $\ph M$ has finite flat dimension or finite injective dimension, then $R$ is regular.  We can rephrase this result in terms of vanishing of derived functors as follows: 

\begin{theorem} (\cite[Theorem 1.1]{AHIY})
\label{ahiy}
Let $(R,\fm, k)$ be a $d$-dimensional Noetherian local ring with maximal ideal $\fm$ and residue field $k$. Let $\phi:R\to R$ be a contracting endomorphism and $M$ a finitely generated nonzero $R$-module.
The following  are equivalent:
\begin{enumerate}[(a)]
\item $R$ is regular.
\item $\Tor_i^R(\ph M, k)=0$ for some (equivalently, every) $i>d$.
\item $\Ext_R^i(\ph M,k)=0$ for some (equivalently, every) $i>d$.
\item $\Ext^i_R(k, \ph M)=0$ for some (equivalently, every) $i>d$.
\end{enumerate}
\end{theorem}
\begin{proof}Clearly, (a) implies both (b), (c), and (d) as regular local rings have finite global dimension.  If (b) holds, then by a result of M. Andre \cite[Lemme 2.57]{An} $\ph M$ has finite flat dimension and $R$ is regular by \cite[Theorem 1.1]{AHIY}.   If (c) holds, then by the proof of \cite[Proposition 5.5P]{AF1991} we see that $\Tor_i^R(\ph M, k)=0$ for some (equivalently, every) $i>\dim R$ and $R$ is regular by part (a).  Finally, suppose (d) holds.   Then by \cite[Proposition 3.2]{CIM}, $\Ext^j_R(k,\ph M)=0$ for all $j\ge i$.  This implies by \cite[Proposition 5.5I]{AF1991} that $\ph M$ has finite injective dimension, and $R$ is regular by \cite[Theorem 1.1]{AHIY}.
\end{proof}

One of our aims is to find results similar to this for Gorenstein rings.  We are unable to find such a result which applies to all contracting endomorphisms, but only for those in which the image of the maximal ideal lies in a sufficiently high power of itself.   One power which suffices is given by a constant first introduced in \cite{AHIY} and later modified in \cite{DIM}.   This constant, which we denote by $c(R)$, is defined in terms of derived Loewy lengths of Koszul complexes of systems of parameters of $R$ (see Definition \ref{homotopical}).    

Below is one of the characterizations we are able to prove (Corollary \ref{cor-main}); additional ones are given in Section 3.

\begin{theorem} \label{intro-theorem}
Let $(R, \fm, k)$ be a $d$-dimensional Noetherian local ring and $\phi:R\to R$ a contracting endomorphism such that $\phi(\fm)\subseteq \fm^{c(R)}$.   Let $M$ be an $R$-module such that $M\neq \fm M$ and $E$ the injective envelope of $k$.
The following conditions are equivalent:
\begin{enumerate}[(a)]
\item $R$ is Gorenstein.
\item $\Tor_i^R(\ph M, E)=0$ for $d+1$ consecutive (equivalently, every) $i>d$
\item $\Ext^i_R(\ph M, R)=0$ for $d+1$ consecutive (equivalently, every) $i>d$.
\item $\Ext^i_R(E, \ph M)=0$ for $d+1$ consecutive (equivalently, every) $i>d$.
\end{enumerate}
\end{theorem}

Our methods require $d+1$ consecutive vanishings of Tor (respectively, Ext), rather than a single vanishing as in Theorem \ref{ahiy}, although we know of no examples where a single vanishing in degree greater than the dimension does not suffice to imply $R$ is Gorenstein.

Finally, we prove another characterization (Proposition \ref{le:prop7}) for Gorenstein rings which uses contracting endomorphisms in a different way:

\begin{proposition}
Let $(R,\fm,k)$ be a Cohen-Macaulay local ring possessing a canonical module $\omega_R$.  Let $\phi:R\to R$ be a contracting homomorphism and let $S$ denote the ring $R$ viewed as an $R$-algebra via $\phi$.  The following are equivalent:
\begin{enumerate}[(a)]
\item $R$ is Gorenstein.
\item $S\otimes_R \omega_R$ has finite injective dimension as an $S$-module.
\end{enumerate}
\end{proposition}

We prove most of our results in the context of complexes where the arguments are more transparent as well as being more general.  For notation and terminology regarding complexes and the derived category, we refer the reader to \cite{AF1991} or \cite{FI}. 

\section{Preliminaries}

In this section we summarize results on homotopical Loewy length, flat dimension, and injective dimension which will be needed in Section 3.    Throughout  this section $(R,\fm,k)$ denotes a local (which also means commutative and Noetherian) ring of arbitrary characteristic, with maximal ideal $\fm$ and residue field $k$.   We let $\sfD(R)$ denote the derived category of $R$-modules,  and `$\simeq$' means an isomorphism in $\sfD(R)$.   

The \emph{Loewy length} of an $R$-complex $L$ is the number
\[
\lol_R(L):= \inf\{n\in \mathbb N\mid \fm^n L=0\}.
\]
The \emph{homotopical Loewy length} of an $R$-complex $L$  is defined as
\[
\lol_{\sfD(R)}(L):= \inf \{\lol_R(V) \mid L \simeq V \text{ in } \sfD(R)\}.
\]

\begin{remark} If $c=\lol_{\sfD(R)}(L)$ then $\fm^c \hh(L)=0$.
\end{remark}

Given a finite sequence $\bsx\in R$ and an $R$-complex $L$, we write $\kos{\bsx}L$ for the Koszul complex on $\bsx$ with coefficients in $L$.    
\begin{proposition}
\label{pr:loewy}
Let $\bsx$ be a finite sequence in $R$ such that the ideal $(\bsx)$ is $\fm$-primary. For each $R$-complex $L$ there are inequalities
\[
\lol_{\sfD(R)} \kos{\bsx}L \leq \lol_{\sfD(R)}\kos{\bsx}R <\infty\,.
\]
\end{proposition}

\begin{proof}  See \cite[Proposition 2.1]{DIM}.
\end{proof}

We will need the following invariant defined in \cite{DIM}:

\begin{definition} \label{homotopical}  For a local ring $R$, set
\[
c(R):=\inf\{\lol_{\sfD(R)} \kos{\bsx}R \mid\text{$\bsx$ is an s.o.p.\,for $R$}\}.
\]
\end{definition}
\noindent From Proposition~\ref{pr:loewy}  it follows that $c(R)$ is finite for any $R$.  

%
%

We'll also need the following proposition:

\begin{proposition}
\label{pr:tor-iso} 
\pushQED{\qed} Let $\phi:(R,\fm,k)\to (S,\fn, l)$ be a homomorphism of local rings and
$\bsy$ be a system of parameters for $S$.  Suppose $\phi(\fm)\subseteq \fn^c$ where $c:=\lol_{\sfD(S)}\kos{\bsy}S$.   Then for each $R$-complex $L$ and $S$-complex $M$, there exist isomorphisms of graded $k$-vector spaces:
\begin{enumerate}[(a)]
\item $\Tor^R_*(\kos{\bsy}M, L) \cong \hh_*(\kos{\bsy}M)\otimes_k \Tor^R_*(k,L)$
\item $\Ext^*_R(\kos{\bsy}M, L)\cong \Hom_k (\hh_*(\kos{\bsy}M), \Ext^*_R(k, L))$.
\item $\Ext^*_R(L, \kos{\bsy}M)\cong \Hom_k (\Tor^R_*(k, L), \hh_*(\kos{\bsy}M))$.
\end{enumerate}
\begin{proof} Part (a) is \cite[Proposition 4.3(2)]{AHIY}.  For (b), 
let $K$ denote $\kos{\bsy}M$.  By \cite[Proposition 4.3(1)]{AHIY}, $K\simeq \hh (K)$ in $\sfD(R)$.  Note that $\hh(K)$ is a $k$-complex, since $\fm S\subseteq \fn^c$ and $\fn^c \hh(K)=0$ by Proposition \ref{pr:loewy}. Thus we have the following isomorphisms in $\sfD(R)$:
\begin{align*}
\RHom_R(K,L)&\simeq \RHom_R (\hh (K), L)\\
&\simeq \RHom_R (\hh (K)\lotimes_k k, L)\\
&\simeq \RHom_k (\hh (K), \RHom_R(k,L))\\
&\simeq \Hom_k(\hh (K), \Ext^*_R(k,L)),
\end{align*}
where the last isomorphism follows as any $k$-complex is isomorphic to its homology in $\sfD(k)$.  Taking homology of the left-hand side, we get the desired result.

Part (c) is proved similarly.
\end{proof}
\end{proposition}

\begin{lemma}
\label{le:prop1}  
Let $\vf\colon (R,\fm,k) \to (S,\fn,l)$ be a homomorphism of local rings such that $\vf(\fm) \subseteq \fn^{c(S)}$. 
Let $L$ be an $R$-complex and $M$ an $S$-complex such that $v:=\inf \hh(M)$ is finite and $\hh_v(M)\otimes_S S/\fn \neq 0$.  

\begin{enumerate}[(a)]
\item If there is an integer $t$ such that $\Tor_{i}^R(M,L)=0$ for $t\le i\le t+\dim S$, then 
$\Tor_{t+\dim S-v}^R(k,L)=0$.
\item If there is an integer $t$ such that $\Ext^{i}_R(M,L)=0$ for $t\le i\le t+ \dim S$, then 
$\Ext_R^{t+\dim S-v}(k,L)=0$.
\item If there is an integer $t$ such that $\Ext^{i}_R(L,M)=0$ for $t\le i\le t+ \dim S$, then 
$\Tor_{t+v}^R(k,L)=0$.
\end{enumerate}
\end{lemma}

\begin{proof}  Assume the hypotheses in (b) hold.
Let $d:=\dim S$ and $\bsy$ an s.o.p for $S$ such that $c(S)=\lol_{\sfD(S)}\kos {\bsy}S$.  Let $K=K[\bsy; M]$. A standard computation for Koszul complexes shows that $\Ext^{t+d}_R(K, L)=0$.  It then follows from Proposition \ref{pr:tor-iso}(b) that $\Ext_R^{t+d-v}(k,L)=0$, since $\hh_{v}(K)\neq 0$ (cf. \cite[1.1]{FI}).

Parts (a) and (c) are proved similarly.
\end{proof}

We recall the following proposition from \cite{CIM}:  
\begin{proposition}
\label{le:prop2}
Let $(R,\fm,k)$ be a local ring and $L$ an $R$-complex.  
\begin{enumerate}[(a)]
\item If $\Tor_n^R(k,L)=0$ for some $n\ge \dim R+\sup \hh(L)$, then $\Tor_i^R(k,L)=0$ for all $i\ge n$.
\item If $\Ext^n_R(k,L)=0$ for some $n\ge  \dim R-\inf \hh(L)$, then $\Ext^i_R(k,L)=0$ for all $i\ge n$.
\end{enumerate}
\end{proposition}
\begin{proof} Part (a) is \cite[Theorem 4.1]{CIM}.  Part (b) follows from \cite[(2.7) and Proposition 3.2]{CIM}.
\end{proof}

For a complex $L$ we let $\fd_R L$ and $\id_R L$ denote the flat dimension and injective dimension, respectively, of $L$.  We'll need the following results from \cite{AF1991}:

\begin{proposition}
\label{le:prop3}
Let $(R, \fm, k)$ be a Noetherian local ring and $L$ an $R$-complex such that $\hh(L)$ is bounded and either $\hh(L)$ is finitely generated or $\Supp_R \hh(L)=\{\fm\}$.
\begin{enumerate}[(a)]
\item $\fd_R L = \sup \{ i\mid \Tor_i^R(k, L)\neq 0\}$.
\item $\id_R L = \sup \{ i\mid \Ext^i_R(k,L)\neq 0\}$.
\end{enumerate}
\end{proposition}
\begin{proof} Part (a) is \cite[Propositions 5.3F and 5.5F]{AF1991}, and part (b) is \cite[Propositions 5.3I and 5.5I]{AF1991}.
\end{proof}

The following result plays a central role in our characterizations of Gorenstein rings:

\begin{proposition}
\label{le:prop4}
Let $(R,\fm,k)$ be a local ring.  The following are equivalent:
\begin{enumerate}[(a)]
\item $R$ is Gorenstein.
\item For every bounded $R$-complex $L$ in $\sfD(R)$, $\fd_R L<\infty$ if and only if $\id_R L<\infty$.
\item There exists an $R$-complex $L$ such that $k\lotimes_R L\not\simeq 0$ with $\fd_R L<\infty$ and $\id_R L<\infty$.
\end{enumerate}
\end{proposition}
\begin{proof}
$(a)\Rightarrow (b)$ follows from \cite[Theorem 3.2]{AF1997}. For $(b)\Rightarrow (c)$, simply let $L=R$.  And $(c)\Rightarrow (b)$ follows from \cite[Proposition 2.10]{F}.
\end{proof}

\begin{remark} \label{supp} For any $R$-complex $L$,  $k\lotimes_R L\not\simeq 0$ if and only if $\RHom_R(k,L)\not\simeq 0$ (cf. \cite[2.1 and 4.1]{FI}).  
\end{remark}

\section{Applications to contracting endomorphisms and Gorenstein rings}

We now apply the results in section 2 to contracting endomorphisms.  An endomorphism $\phi$ of a local ring $(R,\fm, k)$ is called {\it contracting} if $\phi^i(\fm)\subseteq \fm^2$ for some $i$. The Frobenius map, in the case $R$ has prime characteristic, is an example of a contracting endomorphism.  By Cohen's Structure Theorem, any complete equicharacteristic local ring admits a contracting endomorphism, since $R$ contains a copy of its residue field (e.g., the composition $R\to k\hookrightarrow R$ is contracting).  Given a contracting endomorphism $\phi$ and an $R$-complex $M$, we let $\ph M$ denote the complex $M$ viewed as an $R$-complex via $\phi$.

\begin{proposition} Let $\phi$ be a contracting endomorphism of a $d$-dimensional local ring $(R,\fm, k)$ such that $\phi(\fm)\subseteq \fm^{c(R)}$. Let $L$ be an $R$-complex such that $\hh(L)$ is bounded and either $\hh(L)$ is finitely generated or $\Supp_R \hh(L)=\{\fm\}$.    Let $M$ be a $R$-complex such that $v:=\inf \hh(M)$ is finite and $\hh_v(M)\otimes_R R/\fm\neq 0$.
\label{le:prop5} 
\begin{enumerate}[(a)]
\item If there exists an integer $t\ge \inf \hh(M)+\sup \hh(L)$ such that $\Tor_i^R(\ph M, L)=0$ for $t\le i\le t+d$ then $\fd_RL<t+d-\inf \hh(M)$.
\item If there exists an integer $t\ge \inf \hh(M)-\inf \hh(L)$ such that $\Ext^i_R(\ph M, L)=0$ for $t\le i\le t+d$ then $\id_R L<t+d-\inf \hh(M)$.
\item If there exists an integer $t\ge d+\sup \hh(L)-\inf \hh(M)$ such that $\Ext^i_R(L, \ph M)=0$ for $t\le i\le t+d$ then $\fd_R L< t+\inf \hh(M)$.
\end{enumerate}
\end{proposition}
\begin{proof}  For part (a), let $(S, \frak n)$ be the local ring $(R,\fm)$ viewed as an $R$-algebra via $\phi$ and let $d=\dim R$.   Viewing $\ph M$ as an $S$-complex and applying Lemma \ref{le:prop1}(a), we obtain that $\Tor_{t+d-v}^R(k,L)=0$. Using Proposition \ref{le:prop2}(a) and the assumption on $t$, we have $\Tor_i^R(k,L)=0$ for all $i\ge t+d-v$.  Hence, $\fd_RM< t+d-v$ by Proposition \ref{le:prop3}.  
Parts (b) and (c) are proved similarly.
\end{proof}

\begin{theorem}
\label{main-theorem1} Let $\phi$ be a contracting endomorphism of a $d$-dimensional local ring $(R,\fm, k)$ such that $\phi(\fm)\subseteq \fm^{c(R)}$.  Let $L$ be an $R$-complex such that $\hh(L)$ is nonzero, bounded, and either $\hh(L)$ is finitely generated or $\Supp_R \hh(L) = \{\fm\}$,  Let $M$ be an $R$-complex such that $v:=\inf \hh(M)$ is finite and $\hh_v(M)\otimes_R R/\fm\neq 0$.
Suppose one of the following conditions hold:
\begin{enumerate}[(a)]
\item $\id_R L<\infty$ and $\Tor_i^R(\ph M, L)=0$ for $d+1$ consecutive values of $i\ge \inf \hh(M) + \sup \hh(L)$.
\item $\fd_R L<\infty$ and $\Ext^i_R(\ph M, L)=0$ for $d+1$ consecutive values of $i\ge \inf \hh(M)-\inf \hh(L)$.
\item $\id_R L<\infty$ and $\Ext^i_R(L, \ph M)=0$ for $d+1$ consecutive values of $i\ge d+\sup \hh(L) - \inf \hh(M)$.
\end{enumerate}
Then $R$ is Gorenstein.
\end{theorem}
\begin{proof}  For part (a), note that the hypotheses on $L$ imply $k\lotimes_R L\neq 0$ (see Remark \ref{supp}). 
The consecutive $\Tor$ vanishings imply that $L$ has finite flat dimension by  part (a) of Proposition \ref{le:prop5}.   By Proposition \ref{le:prop4}, we obtain that $R$ is Gorenstein.  Parts (b) and (c) are argued similarly.  \end{proof}

For a local ring $R$ let $E_R$ denote the injective envelope of its residue field.   If $R$ is Gorenstein then $\pd_R E_R=\fd_R E_R=\id _R R=\dim R$.  (The first equality is by \cite[Second partie, Th\'eor\`em 3.2.6]{RG}.) Hence, Theorem \ref{main-theorem1} yields the following characterization of Gorenstein rings:

\begin{corollary} \label{cor-main}
Let $(R, \fm, k)$ be a $d$-dimensional local ring and $\phi:R\to R$ a contracting endomorphism such that $\phi(\fm)\subseteq \fm^{c(R)}$.   Let $M$ be an $R$-module such that $M\neq \fm M$. The following conditions are equivalent:
\begin{enumerate}[(a)]
\item $R$ is Gorenstein.
\item $\Tor_i^R(\ph M, E_R)=0$ for $d+1$ consecutive (equivalently, every) $i>d$
\item $\Ext^i_R(\ph M, R)=0$ for $d+1$ consecutive (equivalently, every) $i>d$.
\item $\Ext^i_R(E_R, \ph M)=0$ for $d+1$ consecutive (equivalently, every) $i>d$.
\end{enumerate}
\end{corollary}

We now specialize our results to the case where $R$ has prime characteristic $p$ and the contracting endomorphism is the Frobenius map.  Let $f:R\to R$ be the Frobenius endomorphism and $M$ an $R$-module.  For an integer $e\ge 1$ we let $\ee M$ denote the $R$-module $^{f^e}\!M$.   

First we prove partial converses to Proposition \ref{le:prop5} for the Frobenius map.  We remark that part (a) in the case $M=R$ and $L$ a finitely generated $R$-module is a classic result of Peskine-Szpiro \cite[Th\'eor\`em 1.7]{PS}.  This was later generalized to the case $L$ is an arbitrary module in \cite[Theorem 1.1]{MW}, and to the case $L$ is an arbitrary complex in  \cite[Theorem 1.1]{DIM}.  Part (b) was proved in \cite[Th\'eor\`em 4.15]{PS} in the case $M=R$ and $L$ a finitely generated module, and in \cite[Corollary 3.5]{MW} in the case $M=R$ and $L$ an arbitrary module.

\begin{proposition}
\label{le:prop6}
Let $(R,\fm,k)$ be a Noetherian ring of prime characteristic and let $L$ and $M$ be  $R$-complexes.
\begin{enumerate}[(a)] 
\item If $\fd_R L<\infty$ then $\Tor_i^R(\ee M,L)=0$ for all $i>\fd_R M+\sup \hh(L)$ and $e>0$.
\item Suppose $R$ is $F$-finite, $\hh(M)$ is bounded below and degreewise finitely generated, and $\hh(L)$ is bounded.  If $\id_R L<\infty$ then $\Ext^i_R(\ee M, L)=0$ for all $i>\fd_R M-\inf \hh(L)$ and $e>0$.
\end{enumerate}
\end{proposition}
\begin{proof}  The proof of part (a) largely follows that of \cite[Theorem 1.1: $(1)\Rightarrow (2)$]{DIM}.    Let $e>0$ be an integer and assume $\fd_R L<\infty$, in which case we have $s:=\sup \hh(L)<\infty$.  If $\fd_R M=\infty$ there is nothing to prove, so we assume $t:=\fd_RM <\infty$.  Hence, $\sup \hh(M)=\sup \hh(\ee M)<\infty$; consequently, by \cite[1.5]{FI},  $r:= \sup \hh(\ee M\lotimes_R L)\le \sup \hh(\ee M)+\fd_R L<\infty$. We need to show $r\le t+s$. Again, if $r=-\infty$, there is nothing to prove, so we assume $r$ is finite.  Let $\fp\in \Ass_R \Tor_r(\ee M, L)$.  Since $\fd_{R_{\fp}}M_{\fp}\le t$ and $\sup\hh(L_{\fp})\le s$ , it suffices to prove $r\le t+s$ in the case $(R,\fm,k)$ is local and $\fm\in \Ass_R \Tor_r(\ee M, L)$.   By \cite[2.7]{FI}, we have $\depth \ee M\lotimes_R  L=-r$.   Since $\depth_R M=\depth_R \ee M$, we have by \cite[Theorem 2.1]{I} that $-r=\depth \ee M\lotimes_R L=\depth M\lotimes_R L$.   By \cite[2.7 and 1.5]{FI}, $\depth M\lotimes_R L\ge -\sup \hh(M\lotimes_R L)\ge  -s - t$.  This gives the desired result.

For part (b), as $\ee R$ is a finitely generated $R$-module, we have  that $\hh(\ee M)$ is degreewise finitely generated. If $\fd_RM=\infty$ or $L\simeq 0$ in $\sfD(R)$ there is nothing to prove. Hence, we may assume $\fd_RM<\infty$ and $\inf \hh(L)<\infty$.  Hence  $\hh(M)$, and thus $\hh (\ee M)$, is bounded.  Consequently, $\ee M$ is isomorphic in $\sfD (R)$ to a bounded below complex of finitely generated projective $R$-modules.  By \cite[Lemma 4.4F]{AF1991}, $\Ext^i_R(\ee M, L)_\fm \cong \Ext^i_R(\ee M_{\fm}, L_{\fm})$ for all $i$ and every maximal ideal $\fm$.  Hence, without loss of generality, we may assume $(R,\fm)$ is a local ring.   By \cite[Lemma 4.4I]{AF1991},
$\Ext^i_R(\ee M, L)^{\vee}\cong \Tor_i^R(\ee M, L^{\vee})$ for all $i$, where $(-)^{\vee}:=\Hom_R(-, E_R)$.  Furthermore, $\fd_RL^{\vee}=\id_R L<\infty$ and
$\sup \hh(L^{\vee})=-\inf \hh(L)$.   The result now follows from part (a).
\end{proof}

Before stating the next corollary, we set some terminology and notation.  Again, we let $(R, \fm, k)$ be a local ring.  Recall (e.g., \cite{AF1991}) that an $R$-complex $D$ is {\it dualizing} if $\hh(D)$ is bounded and the homothety morphism $R\to \RHom_R(D,D)$ is an isomorphism in $\sfD(R)$; such a complex exists if and only if $R$ is the quotient of a Gorenstein ring (\cite[Theorem 1.2]{K}).   We say that a dualizing complex $D$ is {\it normalized} if $\sup \hh (D)=\dim R$.  A normalized dualizing complex for $R$, if it exists, is unique up to isomorphism in $\sfD(R)$ and will be denoted by $D_R$.   When $R$ is Cohen-Macaulay and a quotient of a Gorenstein ring, we let $\omega_R$ denote the canonical module of $R$, in which case $D_R \simeq \susp^d \omega_R$ in $\sfD(R)$.  If $\mathbf x$ is a system of parameters for $R$, we let $C(\mathbf x)$ denote the \v Cech complex of $R$ on $\mathbf x$, with the convention that $\inf \hh (C(\mathbf x))=-\dim R$.  

We now highlight several (of many possible) consequences of Theorem \ref{main-theorem1} and Proposition \ref{le:prop6}:

\begin{corollary}
\label{cor1}
Let $(R,\fm,k)$ be a $d$-dimensional local ring of prime characteristic $p$, $\mathbf x$ a system of parameters for $R$,  and $e\ge \log_p c(R)$ an integer.   The following conditions are equivalent:
\begin{enumerate}[(a)]
\item $R$ is Gorenstein.
\item $\Tor_i^R(\ee R, E_R)=0$ for $d+1$ consecutive (equivalently, every) $i>0$.
\item $R$ is a quotient of a Gorenstein ring and $\Tor^R_i(\ee R, D_R)=0$ for $d+1$ consecutive (equivalently, every) $i> \dim R$.
\item $R$ is Cohen-Macaulay, a quotient of a Gorenstein ring,  and $\Tor^R_i(\ee R, \omega_R)=0$ for $d+1$ consecutive (equivalently, every) $i>0$.
\end{enumerate}
\noindent
If in addition $R$ is $F$-finite, the above statements are equivalent to:
\begin{enumerate}[(e)]
\item[(e)] $\Ext^i_R(\ee R, R)=0$ for $d+1$ consecutive (equivalently, every)  $i>0$.
\item[(f)] $\Ext^i_R(\ee R, C(\mathbf x))=0$ for $d+1$ consecutive (equivalently, every) $i>\dim R$.
\end{enumerate}
\end{corollary}
\begin{proof}  Note that if $R$ is Gorenstein, $E_R\cong H^d_m(R)\simeq \susp^{d} C(\mathbf x)$ has finite flat dimension, and $R\cong \omega_R\simeq \susp^{-d} D_R$ has finite injective dimension.  Also, $f^e(\fm)=\fm^{[p^e]}\subseteq \fm^{c(R)}$. The equivalences now follow from Theorem \ref{main-theorem1} and Proposition \ref{le:prop6}.
\end{proof}

We prove one final criterion for Gorensteinness of a somewhat different flavor. 

 Let  $(R,\fm, k)$ be a local ring and $\phi:R\to R$ be a contracting endomorphism.  Endow $\ph R$ with an $R-R$ bimodule structure given by $r\cdot s:=\phi(r)s$ and $s\cdot r:=sr$ for all $r\in R$ and $s\in  \ph R$.    Then $F^{\phi}(-):= -\otimes_R \ph R$  is an additive right exact endofunctor on the category of right $R$-modules.    For a finitely generated $R$-module $M$, we let $\mu_R(M)$ denote the minimal number of generators of $M$.
 
 \begin{lemma} \label{le2}  Let $R$ and $\phi$ be as above and $M$ a finitely generated $R$-module.   Then:
 \begin{enumerate}[(a)]
 \item $F^{\phi}(R^n)\cong R^n$ for every $n\ge 1$.
 \item $\mu_R(F^{\phi}(M))=\mu_R(M)$.
 \item  $\Supp_R F^{\phi}(M)=\{\fp\in \Spec R \mid \phi^{-1}(\fp)\in \Supp_R M\}$.
  \item If $F^{\phi}(M)\cong M$ then $M$ is free.
 \end{enumerate}
 \end{lemma}
 \begin{proof}
 Part (a) is clear, as $F^{\phi}$ is additive and  $\ph R\cong R$ as right $R$-modules.  Part (b) follows by noting that $F^{\phi}$ applied to a minimal presentation of $M$ yields a minimal presentation of $F^{\phi}(M)$.   For part (c), by localizing $R$ at $\phi^{-1}(\fp)$, it suffices to show that for any (local) homomorphism of local rings $R\to S$, $M\neq 0$ if and only if $M\otimes_R S\neq 0$.  This follows readily from Nakayama's lemma.  Part (d) is well-known (e.g., \cite[Lemma 2.1(a)]{Ra}).
 \end{proof}
 
 The following result generalizes \cite[Theorem 4.2.8]{W}, where it is proved in the case $R$ is a complete, one-dimensional $F$-pure ring of prime characteristic and $\phi$ is the Frobenius map:

\begin{proposition}
\label{le:prop7}
Let $(R,\fm,k)$ be a Cohen-Macaulay local ring possessing a canonical module $\omega_R$.  Let $\phi:R\to R$ be a contracting homomorphism.   The following are equivalent:
\begin{enumerate}[(a)]
\item $R$ is Gorenstein.
\item $\operatorname{id}_R F^{\phi}(\omega_R)<\infty$.
\end{enumerate}
\end{proposition}
\begin{proof}
The implication $(a)\Rightarrow (b)$ is clear.  Suppose $(b)$ holds.   Then there exists an exact sequence
$$0\to K\to \omega_R^n\to F^{\phi}(\omega_R)\to 0$$
for some $n\ge 1$ and where $K\subseteq \fm\omega_R^n$ (cf. \cite[Theorem 2.1 and Corollary 2.3]{Sh}).    By Lemma \ref{le2}(b) and Nakayama's lemma,  $\mu_R(\omega_R)=\mu_R(F^{\phi}(\omega_R))=\mu_R(\omega_R^n)$ and hence $n=1$.   We claim $K=0$.  If not, let $\frak p\in \Ass_R K$.  Then $\frak p \in \Ass_R \omega_R=\Min_R R$.   By Lemma \ref{le2}(c), $\Supp_R F^{\phi}(\omega_R)=\Supp_R \omega_R=\Spec R$.   Hence, $F^{\phi}(\omega_R)_{\frak p}\neq 0$ and $\id_{R_{\frak p}} F^{\phi}(\omega_R)_{\frak p}<\infty$. Since $\dim R_{\frak p}=0$,  $F^{\phi}(\omega_R)_{\frak p}\cong \omega_{R_{\frak p}}^{\ell}$ for some $\ell\ge 1$. As $F^{\phi}(\omega_R)_{\frak p}$ is a homomorphic image of $(\omega_R)_{\frak p}\cong \omega_{R_{\frak p}}$, we see that $\ell=1$.  Hence, $K_{\frak p}=0$, a contradiction.   Thus, $K=0$ and $F^{\phi}(\omega_R)\cong \omega_R$.
By part (d) of Lemma \ref{le2}, we obtain that $\omega_R$ is a free $R$-module.  Hence, $R$ is Gorenstein.
\end{proof}

\begin{example} Let $R$ be a quasi-Gorenstein local ring of prime characteristic $p$ and which possesses a dualizing complex $D_R$.  Let $f:R\to R$ be the Frobenius endomorphism.  Since $F^{f}(I)\cong I$ for any injective $R$-module (e.g., \cite[Proposition 3.6]{M}), we see that $F^{f}(D_R)$ is a bounded complex of injective $R$-modules.  Hence, $\id_R F^{f}(D_R)<\infty$.  However, $R$ need not be Gorenstein (e.g., see \cite[Theorem 2.11]{A}).
\end{example}

We leave as a question whether a derived analogue of Proposition \ref{le:prop7} holds:

\begin{question}  Let $(R,\fm.k)$ be a  local ring possessing a dualizing complex $D_R$, and $\phi:R\to R$ a contracting endomorphism.   Suppose $\id_R  D_R\lotimes_R \ph R<\infty$.  Must $R$ be Gorenstein?
\end{question}

\end{document}